\documentclass[a4paper,10pt]{scrartcl}
\usepackage[utf8]{inputenc}
\usepackage[T1]{fontenc}
\usepackage[english]{babel}
\usepackage{amsmath,amsthm,amssymb}
\usepackage{geometry}
\usepackage{xcolor}

\usepackage{enumitem}

\usepackage{newunicodechar}
\newunicodechar{α}{\alpha}
\newunicodechar{β}{\beta}
\newunicodechar{γ}{\gamma}
\newunicodechar{δ}{\delta}
\newunicodechar{ε}{\varepsilon}
\newunicodechar{ϵ}{\epsilon}
\newunicodechar{ζ}{\zeta}
\newunicodechar{η}{\eta}
\newunicodechar{θ}{\theta}
\newunicodechar{ϑ}{\vartheta}
\newunicodechar{ι}{\iota}
\newunicodechar{κ}{\kappa}
\newunicodechar{λ}{\lambda}
\newunicodechar{μ}{\mu}
\newunicodechar{ν}{\nu}
\newunicodechar{ξ}{\xi}
\newunicodechar{ο}{\omikron}
\newunicodechar{π}{\pi}
\newunicodechar{ρ}{\rho}
\newunicodechar{σ}{\sigma}
\newunicodechar{τ}{\tau}
\newunicodechar{υ}{\upsilon}
\newunicodechar{φ}{\varphi}
\newunicodechar{ϕ}{\phi}
\newunicodechar{χ}{\chi}
\newunicodechar{ψ}{\psi}
\newunicodechar{ω}{\omega}
\newunicodechar{Γ}{\Gamma}
\newunicodechar{Δ}{\Delta}
\newunicodechar{Θ}{\Theta}
\newunicodechar{Λ}{\Lambda}
\newunicodechar{Ξ}{\Xi}
\newunicodechar{Π}{\Pi}
\newunicodechar{Σ}{\Sigma}
\newunicodechar{Φ}{\Phi}
\newunicodechar{Ψ}{\Psi}
\newunicodechar{Ω}{\Omega}
\newunicodechar{ℝ}{\mathbb{R}}
\newunicodechar{ℤ}{\mathbb{Z}}
\newunicodechar{ℕ}{\mathbb{N}}
\newunicodechar{ℚ}{\mathbb{Q}}
\newunicodechar{ℂ}{\mathbb{C}}
\newunicodechar{→}{\to}
\newunicodechar{∞}{\infty}
\newunicodechar{∈}{\in}
\newunicodechar{∩}{\cap}
\newunicodechar{∪}{\cup}
\newunicodechar{⊂}{\subset}
\newunicodechar{∇}{\nabla}
\newunicodechar{∅}{\emptyset}
\newunicodechar{∀}{\forall}
\newunicodechar{∃}{\exists}
\newunicodechar{∂}{\partial}
\newunicodechar{⇐}{\Leftarrow}
\newunicodechar{⇔}{\LeftRightarrow}
\newunicodechar{⇒}{\Rightarrow}
\newunicodechar{↦}{\mapsto}
\newunicodechar{⊥}{\perp}

\newtheorem{theorem}{Theorem}[section]
\newtheorem{lemma}[theorem]{Lemma}
\newtheorem{cor}[theorem]{Corollary}
\theoremstyle{definition}

\newtheorem{remark}[theorem]{Remark}

\allowdisplaybreaks

\newcommand{\f}[2]{\frac{#1}{#2}}

\newcommand{\Om}{\Omega}
\newcommand{\Ombar}{\overline{\Om}}
\newcommand{\delny}{∂_{ν}}
\newcommand{\amrand}{\big|_{∂\Om}}
\newcommand{\bdryzT}[1][(0,T)]{\big|_{∂\Om\times#1}}
\newcommand{\kl}[1]{\left(#1\right)}
\newcommand{\ddt}{\frac{d}{dt}}
\newcommand{\io}{\int_{\Om}}
\newcommand{\sub}{\subset}
\newcommand{\Lom}[1]{L^{#1}(\Om)}
\newcommand{\Liom}{L^{\infty}(\Om)}
\newcommand{\Tmax}{T_{max}}
\newcommand{\norm}[2][]{\|#2\|_{#1}}
\newcommand{\nn}{\nonumber}
\newcommand{\na}{\nabla}

\date{}

\begin{document}
 \title{Boundedness of solutions to a virus infection model with saturated chemotaxis}
\author{
Bingran Hu\footnote{hubingran@foxmail.com}\\
{\small School of Information and Technology, Dong Hua University,}\\
{\small 200051 Shanghai, P. R. China }
\and
Johannes Lankeit\footnote{jlankeit@math.uni-paderborn.de}\\
{\small Institut f\"ur Mathematik, Universit\"at Paderborn,}\\
{\small 33098 Paderborn, Germany}
 }
%
%
%
\maketitle

\begin{abstract}
\noindent
 We show global existence and boundedness of classical solutions to a virus infection model with chemotaxis in bounded smooth domains of arbitrary dimension and for any sufficiently regular nonnegative initial data and homogeneous Neumann boundary conditions. 
 More precisely, the system considered is 
 \begin{align*}
  u_t&=Δu - ∇\cdot(\f{u}{(1+u)^{α}}∇v) - uw + κ - u, \\
  v_t&=Δv + uw - v,\\
  w_t&=Δw - w + v,
 \end{align*}
with $\kappa\ge 0$, and solvability and boundedness of the solution are shown under the condition that 
\begin{equation*}
\begin{cases}
\alpha > \f12 + \f{n^2}{6n+4}, &\text{if } \quad 1 \leq n \leq 4 \\
\alpha > \f{n}4, &\text{if } \quad n \geq 5.
\end{cases}
\end{equation*}
  \textbf{Math Subject Classification (2010):} 35Q92, 35K57, 35A01, 35A09, 92C17
\end{abstract}

\section{Introduction}
In theoretical immunology 
 it is not uncommon to model the evolution of a virus population by a system of ODEs (\cite{book_nowak,hethcote}). These models already yield key 
 insights into infections (\cite{book_nowak}; for clinical advice based on models of this type, see e.g. \cite{bonhoeffer_MSN}), but by their very nature are ill-suited for gaining spatial information concerning the distribution of infected cells. 
For this reason in an attempt to better understand the formation of patterns 
on the onset of an HIV infection, in \cite{stancevic_amh} 
the following model was proposed (where $κ$, $α$, $β$, $d_{χ}$, $d_v$ and $d_w$ are suitable positive constants): 
\begin{align}\label{samh:system}
 u_t&=\Delta u - d_{χ}\nabla \cdot(u\nabla v) + κ - (κ-1)uw-u,\\
 v_t&=d_v\Delta v+α(uw-v),\nn\\
 w_t&=d_w\Delta w +β(v-w).\nn
\end{align}
Herein, $u$ and $v$ stand for the population density of uninfected and infected cells, respectively, and $w$ is used to describe the concentration of virus particles. 
All three of the populations move around randomly (i.e. diffuse) and decay. The virus is also produced by infected cells and its presence causes healthy cells to be converted into infected cells. Healthy cells are, moreover, produced with a constant rate $κ$.
In addition, in chemotactic response to cytokines emitted by infected cells healthy cells move toward high concentration of those. 
The corresponding cross-diffusive term in \eqref{samh:system} is the key contributor to mathematical challenges already the global existence analysis of \eqref{samh:system} poses. \\

In contrast to the aggregation phenomena described by the famous classical Keller--Segel type models (\cite{keller_segel_70}, see also the surveys \cite{BBTW,horstmann,hillen_painter}), in the present setting a blow-up of solutions is not to be expected according to the biological observations. Motivated from the desire to hence exclude the possibility of blow-up, in \cite[Sec. 8]{stancevic_amh}, the chemotaxis term was substituted by a term essentially of the form of $\na\cdot(\f{u}{1+u} \na v)$. In line with this reasoning, it is the purpose of this article to investigate whether weaker 
changes can have a similar consequence: If we employ chemotaxis terms of the form $\na\cdot(\f{u}{(1+u)^{α}} \na v)$, can we still guarantee (global existence and) boundedness of solutions? More accurately: for which values of $α$ is it possible? \\

It seems appropriate to note that weakening the cross-diffusion is not the only possible change to \eqref{samh:system} that can ensure global existence and boundedness of solutions:
 
In \cite{tao_bellomo}, Bellomo and Tao replaced the conversion term $uw$ by the term $\f{uw}{1+au+bw}$ of Beddington--deAngelis type (\cite{beddington,deAngelis_et_al}) with positive parameters $a$, $b$, and succeeded in proving global existence and boundedness of solutions to the resulting model, as well as their stabilization as $t\to \infty$ for small basic reproduction numbers. For a closely related system, see also \cite{wang_ma_lai}.\\




From a mathematical perspective, one of the most significant differences between \eqref{samh:system} and the well-studied Keller--Segel type models is the presence of a nonlinear production term ($+uw$ in the second equation). While also chemotaxis--consumption models (see e.g. \cite{tao_consumption,lankeit_wang}), 
popular in the context of studies concerning the interaction between chemotactically active bacteria and their fluid environment (cf. e.g. \cite{ctnslogsource} and references therein), feature a nonlinear term, that term there only appears as sink, not as source term, thus favourably factoring into boundedness considerations. 

The mathematically most inconvenient difference to the spatially homogeneous setting apparently lies in the chemotaxis term. 
Let us briefly contemplate why in its presence the source term of the second equation seems more troublesome: In the ODE setting, a Lyapunov function has been found (in \cite{korobeinikov}), essentially solving questions of boundedness and long-time behaviour. Attempts to employ a corresponding functional (or even only a functional involving the same term for the first component) will result in the necessity to deal with a term of the form $\io \f{\na u\cdot\na v}u$, which in part can be estimated by the contribution of the diffusion term, but then requires something to cancel $\io |\na v|^2$. This we can easily provide by adding $\io v^2$ to the functional, whereupon the nonlinear production term raises its head as $\io uvw$ (cf. \eqref{lem:est:v:l2}) and can barely be controlled by $-\io v^2$ and an analogue stemming from the third equation, if $u$ is replaced by a bounded function, cf. \cite{tao_bellomo}, (which basically is what the change to this term in \cite{tao_bellomo} does). 
Platz ist, is witnessed by the fact that the proof of \cite[Lemma 4.3]{tao_bellomo} relies on smallness of the reproduction number. 
If we want to retain the factor $u$, i.e. the true nonlinearity of the term $+uw$, which was originally taken from the standard SIR model (cf. \cite[Sec. 2]{stancevic_amh}), however, similar estimates seem no longer possible. 

Moreover, in contrast to the situation of \cite[Lemma 3.3]{tao_bellomo}, the influence of $u$ on the growth of $v$ renders us unable to gain a priori boundedness information on $v$ and $w$ in arbitrary $\Lom p$ spaces as easily as there. 

Nevertheless, if the chemotactic effects are lessened in the way suggested in \cite[Sec. 8]{stancevic_amh} (but less severely -- in the physically relevant dimensions), it is possible to show global existence and boundedness of solutions despite the nonlinear source term. 
More precisely: Considering the system 
\begin{subequations}\label{sys}
 \begin{align}
  u_t&=Δu - ∇\cdot\kl{\f{u}{(1+u)^{α}}∇v} - uw + κ - u, \label{sys:u}\\
  v_t&=Δv + uw - v, \label{sys:v}\\
  w_t&=Δw - w + v,\label{sys:w}\\
  &\delny u\amrand=\delny v\amrand=\delny w\amrand=0,\label{sys:boundary}\\
  &u(\cdot,0)=u_0, \quad v(\cdot,0)=v_0, \quad w(\cdot,0)=w_0\label{sys:init}
 \end{align}
\end{subequations}
with initial data satisfying 
\begin{subequations}\label{init}
 \begin{align}
  &u_0\in C^0(\Ombar),\label{init:u}\\
  &v_0\in W^{1,∞}(\Om),\label{init:v}\\
  &w_0\in C^0(\Ombar),\label{init:w}\\
  &u_0\geq0,\quad  v_0\geq 0, \quad  w_0\geq 0,\label{init:nonneg}
 \end{align}
\end{subequations}
we shall show the following: 
\begin{theorem}\label{thm:gebd}
Let $\Om\sub ℝ^n$, $n\in ℕ$, be a bounded, smooth domain. Let \\
\begin{equation}\label{cond:alpha}
\begin{cases}
\alpha > \f12 + \f{n^2}{6n+4}, &\text{if } \quad 1 \leq n \leq 4 \\[0.25cm]
\alpha > \f{n}4, &\text{if } \quad n \geq 5
\end{cases}
\end{equation}
and let $κ \geq 0$. Then, for every $(u_0,v_0,w_0)$ as in \eqref{init}, the problem \eqref{sys} has a global solution which is bounded in the sense that there exists $C>0$ fufilling
\begin{align}\label{thm:boundednessstatement}
\norm[\Liom]{u(\cdot,t)}+\norm[W^{1,∞}(\Om)]{v(\cdot,t)}+\norm[\Liom]{w(\cdot,t)} \leq C   \qquad \text{for all } t>0.
\end{align}
\end{theorem}

\begin{remark}
 For simplicity and better readability we have set most coefficients in model \eqref{sys} equal to $1$. Replacing them by other positive constants would not affect the proofs in any major way.  
 Realistic parameter values can be found in \cite[Sec. 9]{stancevic_amh}.
\end{remark}

The proof consists of two main parts, corresponding to Sections \ref{sec:apriori-locex} and \ref{sec:quasienergy}. In the first of these, we will establish local existence of solutions and show that in order to obtain boundedness in the sense of \eqref{thm:boundednessstatement}, it suffices to bound $t\mapsto \norm[\Lom p]{u(\cdot,t)}$ for some $p>\f n2$. The second part will be devoted to the derivation and use of a quasi-energy inequality, resulting in the confirmation that such an $\Lom p$-norm of $u$ can be controlled, which, in light of the first part, will lead to the conclusion of Theorem \ref{thm:gebd}.

\section{Local existence and a priori estimates}\label{sec:apriori-locex}

We postpone a sketch of the proof of local existence until Lemma \ref{lem:improve:locex}, where we will be able to give a more useful extensibility criterion for the solution than would be possible now, and begin the course of our proofs with the following \textit{a priori} estimates.

\subsection{$L^1$-boundedness of $u$}
The first observation in this direction 
is that the amount of healthy cells remains bounded.
\begin{lemma}\label{lem:u:l1}
 Let $\Om\sub ℝ^n$, $n\in ℕ$, be a bounded, smooth domain. Let $α\geq 0$, $κ\geq0$, $T\in(0,∞]$ and let $0\leq u\in C^0(\Ombar\times[0,T))\cap C^{2,1}(\Ombar\times(0,T))$ solve \eqref{sys:u}, \eqref{sys:boundary} for some nonnegative $w\in C^0(\Ombar\times[0,T))$ and some $v\in C^1(\Ombar\times(0,T))\cap C^0(\Ombar\times[0,T))$ satisfying $\delny v\bdryzT=0$.
 Then
 \begin{equation}\label{bd:u:l1}
  \io u(\cdot,t) \leq e^{-t} \io u(\cdot,0) + κ|\Om| \kl{1-e^{-t}} \qquad \text{for all } t\in (0,T).
 \end{equation}
\end{lemma}
\begin{proof}
 Integrating \eqref{sys:u} and using nonnegativity of $uw$, one obtains
 \[
  \ddt \io u(\cdot,t) \leq κ|\Om|-\io u(\cdot,t)\qquad \text{for all } t\in(0,T)
 \]
 so that \eqref{bd:u:l1} follows.
\end{proof}

\subsection{$L^1$-boundedness of $u+v$ and of $v$}
The next step should be the derivation of similar $L^1$ bounds for $v$. However, the nonlinear production $+uw$ blocks similarly simple approaches as employed in the proof of Lemma \ref{lem:u:l1}. We therefore first turn our attention to the total mass of healthy and infected cells, where the terms modelling the conversion of healthy to infected cells in each of the respective equations cancel each other.

\begin{lemma}\label{lem:uplusv:l1}
 Let $\Om\sub ℝ^n$, $n\in ℕ$, be a bounded, smooth domain. Let $α\geq 0$, $κ\geq 0$, $T\in(0,∞]$ and let the nonnegative functions $u, v\in C^0(\Ombar\times[0,T))\cap C^{2,1}(\Ombar\times(0,T))$ solve \eqref{sys:u}, \eqref{sys:v}, \eqref{sys:boundary} for some function $w$. Then
\begin{equation}\label{eq:uplusv:l1}
 \io u(\cdot,t)+\io v(\cdot,t) = e^{-t} \kl{\io u(\cdot,0)+\io v(\cdot,0)} + κ|\Om| (1-e^{-t}) \qquad \text{for all } t\in (0,T).
\end{equation}
\end{lemma}
\begin{proof}
 Adding \eqref{sys:u} and \eqref{sys:v} shows that
\[
 (u+v)_t = Δu + Δv - ∇\cdot\kl{\f{u}{(1+u)^{α}}∇v} + κ - (u+v) \qquad \text{in } \Om\times(0,T)
\]
and integration over $\Om$ leads to the ODE $y'(t)=κ|\Om| - y(t)$, $t\in (0,T)$, for $y(t):=\io u(\cdot,t)+\io v(\cdot,t)$, $t\in (0,T)$, immediately resulting in \eqref{eq:uplusv:l1}.
\end{proof}

We now may conclude that also $t\mapsto \io v(\cdot,t)$ is bounded.

\begin{cor}\label{cor:v:l1}
 Let $\Om\sub ℝ^n$, $n\in ℕ$, be a bounded, smooth domain. Let $α\geq 0$, $κ\geq 0$, $T\in(0,∞]$ and let the nonnegative functions $u, v\in C^0(\Ombar\times[0,T))\cap C^{2,1}(\Ombar\times(0,T))$ solve \eqref{sys:u}, \eqref{sys:v}, \eqref{sys:boundary} for some function $w$. Then
 \[
  \io v(\cdot,t) \leq e^{-t} \kl{\io u(\cdot,0)+\io v(\cdot,0)} + κ|\Om| (1-e^{-t}) \qquad \text{for all } t\in (0,T).
 \]
\end{cor}
\begin{proof}
 Nonnegativity of $u$ and Lemma \ref{lem:uplusv:l1} imply this statement.
\end{proof}

\subsection{Boundedness of $w$ in $L^{\f{n}{n-2}-ε}$}
$L^1$ boundedness of $v$, which serves as the source term in \eqref{sys:w}, already entails some boundedness for the third solution component $w$. As preparation for later arguments, we give the following lemma in a more general form, but note that in light of Corollary \ref{cor:v:l1} can immediately apply it with $q=1$ and hence for any $r\in[1,\f n{n-2})$.

\begin{lemma}\label{lem:improve:w}
 Let $\Om\sub ℝ^n$ be a bounded, smooth domain. Then there is $C>0$ such that for any $q\in[1,\infty)$, 
any $r\in[1,\infty]$ such that $r<\f{nq}{n-2q}$ if $q<\f n2$, $r<\infty$ if $q=\f n2$ or $r\le \infty$ if $q>\f n2$,
 any $T\in(0,∞]$ and any $w\in C^{2,1}(\Ombar\times(0,T))\cap C^0(\Ombar\times[0,T))$ solving \eqref{sys:w}, \eqref{sys:boundary} for some $v\in C^0(\Ombar\times[0,T))\cap C^{2,1}(\Ombar\times(0,T))$ that satisfies
 \[
  \norm[\Lom q]{v(\cdot,t)} \leq K \qquad \text{ for all } t\in(0,T)
 \]
 with some $K>0$,  we have
 \begin{equation}\label{bd:w:n.nminus2}
  \norm[\Lom r]{w(\cdot,t)}\leq C \kl{e^{-t}\norm[\Lom r]{w(\cdot,0)} + K}\qquad \text{for all } t\in (0,T).
 \end{equation}
\end{lemma}
\begin{proof}
If we represent $w$ according to
\[
 w(\cdot,t)=e^{t(Δ-1)}w(\cdot,0) + \int_0^t e^{(t-s)(Δ-1)}v(\cdot,s) ds, \qquad t\in (0,T),
\]
by the variation-of-constants formula, the well-known $L^p$-$L^q$-estimates for the (Neumann-)heat semigroup (see \cite[Lemma 1.3 i)]{win_aggregationvs}) provide us with a constant $k_1>0$ such that
\begin{align*}
 \norm[\Lom r]{w(\cdot,t)}&\leq k_1 e^{-t} \norm[\Lom r]{w(\cdot,0)} + \int_0^t e^{-(t-s)} k_1 \kl{1+(t-s)^{-\f n2(\f1q-\f1r)}} \norm[\Lom q]{v(\cdot,s)} ds\\
 &\leq C\kl{e^{-t}\norm[\Lom r]{w(\cdot,0)} + K}\qquad \text{for all } t\in(0,T),
\end{align*}
where $C:=\max\{k_1, k_1\int_0^{∞} e^{-τ}(1+τ^{-\f n2(\f1q-\f1r)})dτ\}$, which is finite due to the condition $r<\f {nq}{n-2q}$ ensuring $-\f n2(\f1q-\f1r)>-\f n2(\f1q-\f{n-2q}{nq})=-1$.
\end{proof}

\subsection{Conditional regularity of $v$}
Similarly, it is possible to assert boundedness of the second solution component in smaller spaces than given by Corollary \ref{cor:v:l1} -- provided that boundedness of $w$ and $u$ is already known. Like Lemma \ref{lem:improve:w}, the following lemma will become part of iterative procedures later (in the proofs of Lemma \ref{lem:unifbdness:w} and Lemma \ref{lem:u:bd:if:u:nhalf:and:w:bd}).

\begin{lemma}\label{lem:bd:v}
 Let $\Om\sub ℝ^n$ be a bounded, smooth domain. Let $p\in[1,∞]$, $r\in[1,∞]$ be such that $\f1p+\f1r\leq1$ and let $q_1, q_2\in [1,∞]$ satisfy
 \begin{align}\label{eq:conditionsonq}
  \f1p+\f1r-\f1n&<\f1{q_1}\nn\\
  \f1p+\f1r-\f2n&<\f1{q_2}.
 \end{align}
 Then there is $C>0$ such that whenever for some $T\in(0,∞]$, the function $v\in C^0(\Ombar\times[0,T))\cap C^{2,1}(\Ombar\times(0,T))$ satisfies \eqref{sys:v}, \eqref{sys:boundary} for some $u,w\in C^0(\Ombar\times[0,T))$ satisfying
 \begin{align*}
  \norm[\Lom p]{u(\cdot,t)}\leq K \qquad \text{for all } t\in(0,T)\\
  \norm[\Lom r]{w(\cdot,t)}\leq K \qquad \text{for all } t\in(0,T)
 \end{align*}
 for some $K>0$, then
 \begin{align*}
  \norm[\Lom {q_1}]{∇v(\cdot,t)}&\leq C(e^{-t}\norm[\Lom{q_1}]{∇v(\cdot,0)} + K^2) \\
  \norm[\Lom{q_2}]{v(\cdot,t)}&\leq C(e^{-t}\norm[\Lom{q_2}]{v(\cdot,0)}+K^2).
 \end{align*}
\end{lemma}
\begin{proof}
 With $ρ\in[1,∞]$ being defined by $\f1{ρ}=\f1p+\f1r$ we have that $\norm[\Lom{ρ}]{uw(\cdot,t)}\leq K^2$ for any $t\in(0,T)$ and invoking a semigroup representation of $v$ and estimates from \cite[Lemma 1.3]{win_aggregationvs}, we immediately see that for any $t\in(0,T)$ we have
 \[
  \norm[\Lom{q_2}]{v(\cdot,t)}\leq k_1e^{-t}\norm[\Lom{q_2}]{v(\cdot,0)} + k_1K^2\int_0^\infty e^{-τ}\kl{1+τ^{-\f n2(\f1{ρ}-\f1{q_2})}}dτ
 \]
 and that
 \[
  \norm[\Lom{q_1}]{∇v(\cdot,t)}\leq k_1e^{-t}\norm[\Lom{q_1}]{∇v(\cdot,0)} + k_1K^2\int_0^\infty e^{-τ}\kl{1+τ^{-\f12-\f n2(\f1{ρ}-\f1{q_1})}}dτ
 \]
 for any $t\in(0,T)$, where $k_1>0$ is the constant obtained from \cite[Lemma 1.3]{win_aggregationvs} and where the integrals are finite due to \eqref{eq:conditionsonq}.
\end{proof}

\subsection{Boundedness of $w$}
If certain bounds on $u$ are assumed, we are in the following situation: Regularity of $w$ entails regularity of $v$ (according to Lemma \ref{lem:bd:v}), and higher regularity of $v$, in turn, can be used to further our knowledge concerning regularity of $w$ (Lemma \ref{lem:improve:w}). Accordingly, we can iterate application of these two lemmata so as to obtain the following.

\begin{lemma}\label{lem:unifbdness:w}
 Let $\Om\sub ℝ^n$ be a bounded, smooth domain. Let $p\in[1,∞]$, $r\in[1,∞]$ be such that $\f1p+\f1r\leq1$ and let $p>\f n4$. Then for every $K>0$ there is $C>0$ such that the following holds:
 Whenever $T\in(0,∞]$ and $v,w\in C^0(\Ombar\times[0,T))\cap C^{2,1}(\Ombar\times(0,T))$ satisfy \eqref{sys:boundary}, \eqref{sys:v} and \eqref{sys:w} for some $u\in C^0(\Ombar\times[0,T])\cap C^{2,1}(\Ombar\times(0,T))$ and
 \begin{align}\label{eq:conditions:leading:to:better:w}
  \sup_{t\in(0,T)}\norm[\Lom p]{u(\cdot,t)}\leq K,\nn\\
  \sup_{t\in(0,T)}\norm[\Lom r]{w(\cdot,t)}\leq K,\\
  \norm[\Lom\infty]{w(\cdot,0)}\leq K,\nn\\
  \norm[\Lom\infty]{v(\cdot,0)}\leq K\nn
 \end{align}
 are satisfied, then
 \[
  \sup_{t\in(0,T)}\norm[\Lom\infty]{w(\cdot,t)}\leq C.
 \]
\end{lemma}

\begin{proof}
 By Lemma \ref{lem:bd:v}, we have a bound on $\sup_{t\in(0,T)} \norm[\Lom q]{v(\cdot,t)}$ if $\f1p+\f1r-\f2n<\f1q$. According to Lemma \ref{lem:improve:w} this can be turned into a bound on $\sup_{t\in(0,T)}\norm[\Lom {r'}]{w(\cdot,t)}$ if $\f1{r'}>\f1q-\f2n$. Combining these conditions, we see that \eqref{eq:conditions:leading:to:better:w} entails $\sup_{t\in(0,T)} \norm[\Lom {r'}]{w(\cdot,t)}<\infty$ whenever $\f1{r'}>\f1p+\f1r-\f4n$.
 Since $p>\f n4$, we have that $a:=\f 4n-\f1p$ is positive.
 Setting $r_0:=r$ and \[r_{k+1}:=f(r_k):=\begin{cases}\f{2r_k}{2-ar_k},& r_k<\f2a,\\\infty, & r_k\geq \f2a, \end{cases}\] for $k\inℕ_0$ (meaning that $\f1{r_{k+1}}=(\f1{r_k}-\f a2)>\f1{r_k}+\f1p-\f4n$), upon iteration of the previous argument, we obtain $C_k(K)>0$ such that $\sup_{t\in(0,T)}\norm[\Lom {r_k}]{w(\cdot,t)}\leq C_k(K)$. We finally remark that for some finite $k\inℕ$ we have $r_k=\infty$, because $0<r<f(r)$ for $0<r<\f2a$ and $f$ has no fixed point.
\end{proof}

If we account for the bounds on $w$ that we have prepared previously, this lemma immediately implies boundedness of the third solution component:

\begin{cor}\label{cor:w:bd:if:u:ln2}
 Let $\Om\sub ℝ^n$, $n\in ℕ$, be a bounded, smooth domain and let $K>0$ and $p\in(\f n2,∞]$. Then there is $C>0$ such that the following holds:
 If $α\geq 0$, $κ\geq 0$, $T\in(0,∞]$ and the functions $u, v, w\in C^0(\Ombar\times[0,T))\cap C^{2,1}(\Ombar\times(0,T))$ solve \eqref{sys} with some $(u_0,v_0,w_0)$ as in \eqref{init}, and if
 \[
  \sup_{t\in(0,T)} \norm[\Lom p]{u(\cdot,t)}\leq K, \quad \norm[\Lom\infty]{w_0}\leq K,\quad \norm[\Lom\infty]{v_0}\leq K,
 \]
 then
 \begin{equation}\label{eq:w:bounded}
  \sup_{t\in(0,T)}\norm[\Lom\infty]{w(\cdot,t)}\leq C.
 \end{equation}
\end{cor}
\begin{proof}
 We pick $r\in[1,\f{n}{n-2})$ such that $\f1r+\f1p\leq 1$, which is possible due to $p>\f n2$. For this $r$, Lemma \ref{lem:improve:w} together with Corollary \ref{cor:v:l1} shows that $\sup_{t\in(0,T)}\norm[\Lom r]{w(\cdot,t)}<\infty$, so that Lemma \ref{lem:unifbdness:w} becomes applicable and guarantees \eqref{eq:w:bounded}.
\end{proof}

\subsection{Improving regularity of $u$}
The previous lemmata require some bounds concerning $u$. Therefore, our next aim shall be the improvement of boundedness properties of said component. In a first step we use assumed boundedness of $\na v$ to procure estimates of norms of $u$.

\begin{lemma}\label{lem:improve:u}
Let $\Om\sub ℝ^n$ be a bounded, smooth domain. Let $α\geq 0$,  $κ\geq 0$ and let $q_0\in[1,∞]$, $p_0\in [1,∞]$ be such that
\begin{equation}\label{cond:p0}
 p_0\geq \f{(1-α)_+}{1-\f1{q_0}}
\end{equation}
is satisfied. Moreover, let $p\in[1,∞]$ be such that
\begin{equation}\label{eq:pcondition}
 \f1{q_0}+\f{(1-α)_+}{p_0}-\f1n < \f1p.
\end{equation}
Then there is $C>0$ such that whenever for some $T\in(0,∞]$, the nonnegative function $u\in C^0(\Ombar\times[0,T))\cap C^{2,1}(\Ombar\times(0,T))$ solves  \eqref{sys:boundary} and \eqref{sys:u} for some $v\in C^1(\Ombar\times(0,T))$ and some nonnegative function $w\in C^0(\Ombar\times[0,T))$
and
the estimates
\begin{align}\label{eq:sup:u:p0:v:q0}
 \sup_{t\in(0,T)} \norm[\Lom{p_0}]{u(\cdot,t)}\leq K\nn\\
 \sup_{t\in(0,T)} \norm[\Lom{q_0}]{∇v(\cdot,t)}\leq K
\end{align}
are fulfilled for any $K>0$ and \eqref{sys:boundary} holds,
then we have
\[
 \norm[\Lom p]{u(\cdot,t)}\leq C\kl{1+K^{1+(1-α)_+}+e^{-t}\norm[\Lom p]{u(\cdot,0)}} \qquad \text{for all } t\in(0,T).
\]
\end{lemma}
\begin{proof}
 The assumption \eqref{cond:p0} ensures that $1\geq \f1{q_0}+\f{(1-α)_+}{p_0}=:\f1r$ so that Hölder's inequality becomes applicable and guarantees that, due to \eqref{eq:sup:u:p0:v:q0},
 \begin{align}\label{eq:est:lr}
  \norm[\Lom r]{\f{u}{(1+u)^{α}}∇v}&\leq \norm[\Lom {q_0}]{∇v}\norm[\Lom {\f{p_0}{(1-α)_+}}]{u^{(1-α)_+}}\nn\\
  &\leq K^{1+(1-α)_+}\qquad\qquad \text{on } (0,T).
 \end{align}
The variation-of-constants representation of $u$ together with nonnegativity of $uw$ and semigroup estimates ensure that with some $k_1>0$, $k_2>0$ taken from \cite[Lemma 1.3]{win_aggregationvs}, we obtain
\begin{align*}
 \norm[\Lom p]{u(\cdot,t)}&\leq k_1 e^{-t}\norm[\Lom p]{u(\cdot,0)} + κ(1-e^{-t})|\Om|^{\f1p}\\&\qquad +k_2K^{1+(1-α)_+} \int_0^{∞} e^{-τ}(1+τ^{-\f12-\f n2(\f1{q_0}+\f{(1+α)_+}{p_0}-\f1p)})dτ \qquad \text{for any } t\in(0,T)
\end{align*}
by \eqref{eq:est:lr}, where the integral is finite due to \eqref{eq:pcondition}.
\end{proof}

\subsection{Boundedness}
It seems to be a shortcoming of Lemma \ref{lem:improve:u} that it has to rely on some known estimates for $\na v$. On the other hand, from Lemma \ref{lem:bd:v} we already know how to obtain those, if we may assume some boundedness of $w$ (which is unproblematic in view of Corollary \ref{cor:w:bd:if:u:ln2}): 

\begin{lemma} \label{lem:u:bd:if:u:nhalf:and:w:bd}
 Let $\Om\sub ℝ^n$, $n\in ℕ$, be a bounded, smooth domain and let $α\geq 0$, $κ\geq 0$. Moreover let $p_0, p\in[1,∞]$ be such that
 \begin{equation}\label{condonp0}
 p_0>\f{1+(1-α)_+}{1+\f1n}
\end{equation}
and
\begin{equation}\label{cond:p0p}
 \f{1+(1-α)_+}{p_0}-\f2n<\f1p,
\end{equation}
and let $q\in[1,∞]$ fulfil $\f1q>\f1p-\f1n$.
Then for every $K>0$ there is $C>0$ such that whenever for some $T\in(0,∞]$, the functions $u,v\in C^{2,1}(\Ombar\times(0,T))\cap C^0(\Ombar\times[0,T))$ solve \eqref{sys:u}, \eqref{sys:v}, \eqref{sys:boundary} for some function $w\in C^0(\Ombar\times[0,T))$ satisfying
\[
 \sup_{t\in(0,T)} \norm[\Lom\infty]{w(\cdot,t)}\leq K
\]
and
\[
 \sup_{t\in(0,T)}\norm[\Lom {p_0}]{u(\cdot,t)}\leq K, \quad  \norm[\Lom p]{u_0}\le K,\quad \norm[\Lom\infty]{∇v_0}\le K
\]
are satisfied, then
\begin{equation}\label{eq:boundinup}
 \sup_{t\in(0,T)} \norm[\Lom p]{u(\cdot,t)}<C
\end{equation}
and
\begin{equation}\label{eq:boundnav}
 \sup_{t\in(0,T)} \norm[\Lom q]{∇ v(\cdot,t)}<C.
\end{equation}
\end{lemma}
\begin{proof}
 Since \eqref{condonp0} ensures that $p_0>\f{(1+α)_+}{1-\f1{p_0}+\f1n}$ and since \eqref{cond:p0p} is valid, it is possible to find $q_0\in[1,∞]$ such that $\f1{p_0}-\f1n<\f1{q_0}$ and $p_0\geq\f{(1-α)_+}{1-\f1{q_0}}$ and $\f1{q_0}+\f{(1-α)_+}{p_0}-\f1n<\f1p$. Successive application of Lemma \ref{lem:bd:v} (with $r=\infty$) and Lemma \ref{lem:improve:u} therefore yield \eqref{eq:boundnav} and \eqref{eq:boundinup}. 
\end{proof}

In conclusion, we arrive at the statement that the boundedness of the $\Lom p$-norm of the first component for some $p>\f n2$ already entails boundedness of the solution.



\begin{cor}\label{cor:u:lp:gives:boundedness}
Let $\Om\sub ℝ^n$, $n\in ℕ$, be a bounded, smooth domain and let $α\geq 0$, $κ\geq 0$. Let $p\in[1,∞]$ satisfy
\[
 p>\f n2, \quad p>\f{(1-α)_+n}2, \quad p>\f{1+(1-α)_+}{1+\f1n}.
\]
Then for every $K>0$ there is $C>0$ such that whenever $(u,v,w)\in\kl{C^0(\Ombar\times[0,T))\cap C^{2,1}(\Ombar\times(0,T))}^3$ solves \eqref{sys} for initial data as in \eqref{init} and
\[
 \sup_{t\in(0,T)}\norm[\Lom p]{u(\cdot,t)}\leq K, \qquad \kl{\norm[\Lom\infty]{u_0}+\norm[W^{1,∞}(\Om)]{v_0}+\norm[\Lom\infty]{w_0}}\leq K,
\]
then
\[
 \sup_{t\in(0,T)}\kl{\norm[\Lom\infty]{u(\cdot,t)}+\norm[W^{1,∞}(\Om)]{v(\cdot,t)}+\norm[\Lom\infty]{w(\cdot,t)}}\leq C.
\]
\end{cor}
\begin{proof}
 According to Corollary \ref{cor:w:bd:if:u:ln2}, $w$ is bounded in $\Om\times(0,T)$.
 By an iterative procedure similar to that in the proof of Lemma \ref{lem:unifbdness:w}, repeated application of Lemma \ref{lem:u:bd:if:u:nhalf:and:w:bd} shows boundedness of $u$ and of $\nabla v$.
\end{proof}

We close this section with the local existence and extensibility statement we have been working towards:

\begin{lemma}\label{lem:improve:locex}
 Let $\Om\sub ℝ^n$, $n\in ℕ$, be a bounded, smooth domain. Let $α\geq0$, $κ\geq 0$. For every $(u_0,v_0,w_0)$ as in \eqref{init} there are $\Tmax>0$ and a triple of functions
 \begin{align*}
  u&\in C^0(\Ombar\times[0,\Tmax))\cap C^{2,1}(\Ombar\times(0,\Tmax)),\\
  v&\in C^0(\Ombar\times[0,\Tmax))\cap C^{2,1}(\Ombar\times(0,\Tmax)),\\
  w&\in C^0(\Ombar\times[0,\Tmax))\cap C^{2,1}(\Ombar\times(0,\Tmax))
 \end{align*}
 such that $(u,v,w)$ solves \eqref{sys} and that
 \begin{equation}\label{eq:extensibility:better}
  \Tmax=\infty \qquad\text{or}\qquad \lim_{t\nearrow\Tmax} \norm[\Lom p]{u(\cdot,t)}=\infty
 \end{equation}
for any $p>\f n2$ and such that $u$, $v$ and $w$ are nonnegative in $\Om\times(0,\Tmax)$.
\end{lemma}
\begin{proof}
 Standard reasoning along the lines of e.g. \cite[Thm. 3.1]{horstmann_winkler}, with Banach's fixed point theorem residing at its core, yields the existence result and shows that 
\begin{equation}\label{eq:extensibility:worse}
  \text{either }\quad \Tmax=\infty \qquad\text{or}\qquad \lim_{t\nearrow\Tmax} \kl{\norm[\Liom]{u(\cdot,t)}+\norm[W^{1,∞}(\Om)]{v(\cdot,t)}+\norm[\Liom]{w(\cdot,t)}}=\infty. 
 \end{equation}
An application of Corollary \ref{cor:u:lp:gives:boundedness} allows us to replace \eqref{eq:extensibility:worse} by \eqref{eq:extensibility:better}.
\end{proof}

\section{Preparations for a quasi-energy inequality}\label{sec:quasienergy}
The main part in the analysis of solutions to \eqref{sys} will be played by a quasi-energy inequality, a differential inequality for the function 
\[
 \io u^p + \io v^2 + \io |\na w|^2. 
\]
In the first subsections, we will, step by step, prepare differential inequalities for the summands separately.


\subsection{Estimating $\io u^p$}
We begin with estimates for a power of $u$. In the end, we will have to use this for some exponent greater than $\f n2$ -- since this is part of the assumptions in Corollary \ref{cor:u:lp:gives:boundedness} and \eqref{eq:extensibility:better}. 
It is apparent, at least for $n>2$, that merely integrating the ODE Lyapunov functional (\cite{korobeinikov}) over the domain could not possibly entail sufficient boundedness information on $u$, -- and in any case this approach would not account for the chemotaxis term, as discussed in the introduction. 

\begin{lemma}
 Let $\Om\sub ℝ^n$, $n\in ℕ$, be a bounded, smooth domain. Let $(u_0,v_0,w_0)$ satisfy \eqref{init}. Suppose that $\alpha \geq \f12$, $k \geq 0$. Then for all $p \in [1,2\alpha]$, there is $C_1>0$ such that the solution of \eqref{sys} satisfies
\begin{equation}\label{lem:pre:u}
\f1p \ddt\io u^p +\f1p \io u^p \leq \f{p-1}4\io |∇v|^2 - \io u^pw  + C_1
\end{equation}
on $(0,\Tmax)$.
\end{lemma}
\begin{proof}
Multiplying equation \eqref{sys:u} by $u^{p-1}$ and integrating over $\Omega$ by parts, using Young's inequality, we deduce that \\
\begin{align*}
 \f1p \ddt\io u^p &= -(p-1) \io u^{p-2} |∇u|^2 + (p-1) \io \f{u^{p-1}}{(1+u)^{α}} ∇u\cdot ∇v \\&+ κ\io u^{p-1} -\io u^p - \io u^pw \\
 &\leq \f{p-1}4\io \f {u^p}{(1+u)^{2\alpha}}|∇v|^2 - \f1p \io u^p -\io u^pw  + C_1\\
 &\leq \f{p-1}4\io |∇v|^2 - \f1p \io u^p -\io u^pw  + C_1\\
\end{align*}
holds on $(0,\Tmax)$, where $\f {u^p}{(1+u)^{2\alpha}} \leq 1$ due to $p \leq 2\alpha$ and where we have set $C_1:=\f{\kappa^p}{p}|Ω|$.\\
\end{proof}

\subsection{Estimating $\io v^2$}
Before we derive a differential inequality for $\io v^2$, in Lemmata \ref{lem:w:lq} and let us prepare estimates that will allow us to transform the term $+\io uvw$ on its right-hand side into terms we can control.
\begin{lemma}\label{lem:w:lq}
Let $\Om\sub ℝ^n$, $n\in ℕ$, be a bounded, smooth domain. Let $(u_0,v_0,w_0)$ satisfy \eqref{init}. Let $q\in(1,\infty]$ be such that $q\le \infty$ if $n=1$, $q<\infty$ if $n=2$ and $q<\f{2n}{n-2}$ if $n\ge 3$.
Then for any $\varepsilon_1 >0$, there exists $C(\varepsilon_1)>0$ such that the solution of \eqref{sys} satisfies
\begin{equation}\label{eq:est:wq}
\io w^q(\cdot,t) \leq \varepsilon_1 \io |D^2w(\cdot,t)|^2 + C(\varepsilon_1)
\end{equation}
for all $t \in(0,\Tmax)$\\
\end{lemma}
\begin{proof}
Thanks to Lemma \ref{lem:improve:w} in conjunction with Corollary \ref{cor:v:l1} we know that 
\begin{equation}\label{eq:wbd}
 \sup_{t\in(0,\Tmax)} \io w^{ρ}(\cdot,t) <\infty \qquad \text{for } ρ\in[1,\frac{n}{n-2})
\end{equation}
and hence 
only need to discuss the case of $q \geq \f{n}{n-2}$.
The Gagliardo-Nirenberg inequality provides us with $C_2>0$ such that
\begin{align*}
\io w^q = \norm[\Lom q]{w}^q \leq C_2\norm[\Lom 2]{D^2w}^{aq}\cdot \norm[\Lom s]{w}^{(1-a)q} + C_2\norm[\Lom s]{w}^q  \qquad \text{on }(0,\Tmax)
\end{align*}
where $a=\f {\f 1s - \f 1q}{\f 1s +\f 2n - \f 12} \in (0,1)$ and $s \in (1,\f n{n-2})$. That the exponent $aq$ satisfies $aq<2$ is guaranteed by the restriction $q< \f{2n}{n-2}$. Then, due to \eqref{eq:wbd} with $ρ=s$, the proof is complete after applying Young's inequality.
\end{proof}

\begin{lemma}\label{lem:v:lr}
Let $\Om\sub ℝ^n$, $n\in ℕ$, be a bounded, smooth domain. Let $(u_0,v_0,w_0)$ satisfy \eqref{init}. Let $r\in (1,2+\f2n)$. Then for any $\varepsilon_2 >0$, there exists $C(\varepsilon_2) >0$ such that the solution of \eqref{sys} satisfies
\begin{equation}\label{eq:est:vr}
\io v^r \leq \varepsilon_2 \io |∇v|^2 + C(\varepsilon_2)   \qquad \text{on } (0,\Tmax).
\end{equation}
\end{lemma}

\begin{proof}
We can obtain inequality \eqref{eq:est:vr} in a way very similar to Lemma \ref{lem:w:lq}: We note that the condition on $r$ ensures that $a:=\f{1-\f1r}{\f12+\f1n}$ satisfies $a\in(0,1)$ and $ar<2$ and conclude \eqref{eq:est:vr} from Gagliardo--Nirenberg's inequality showing that with some $C_3>0$ 
\[
 \io v^r=\norm[\Lom r]{v}^r \le C_3\norm[\Lom 2]{\nabla v}^{ar}\norm[\Lom 1]{v}^{(1-a)r} + C_3\norm[\Lom 1]{v}^r \qquad \text{on } (0,\Tmax)
\]
and from an application of Young's inequality together with Lemma \ref{cor:v:l1}. 
\end{proof}

\begin{lemma}
 Let $\Om\sub ℝ^n$, $n\in ℕ$, be a bounded, smooth domain. Let $(u_0,v_0,w_0)$ satisfy \eqref{init}. Let $p>1+\f{n^2}{3n+2}$. There exists $C_4>0$ such that the solution of \eqref{sys} satisfies
\begin{equation}\label{lem:pre:v}
\f12\ddt \io v^2 + \io v^2 + \f12\io |∇v|^2 \leq \f2{p+3} \io u^pw + \f2{p+3}\io |D^2w|^2 + C_4
\end{equation}
on $(0,\Tmax)$\\
\end{lemma}

\begin{proof}
Testing equation \eqref{sys:v} against $v$ and integrating by parts over $\Omega$, we obtain that
\begin{align}\label{lem:est:v:l2}
\f12 \ddt \io v^2 = -\io |∇v|^2  - \io v^2 +\io uvw   \qquad \text{on } (0,\Tmax).
\end{align}
Planning to deal with the last term on the right, we observe that $p>1+\f{n^2}{2n+2}$ entails $\frac{p}{p-1}\cdot \frac{\f{2n}{n-2}}{\f{2n}{n-2}-1} < 2+\f2n$ and fix $q\in(1,\f{2n}{n-2})$ such that $r:=\f{p}{p-1}\cdot\f{q}{q-1}\in (1,2+\f2n)$. 
By twofold application of Young's inequality, 
\begin{align}\label{eq:est:vr&wq}
\io uvw &\leq \f2{p+3} \io u^pw + C_5\io v^\f{p}{p-1}w  \nn \\
        &\leq \f2{p+3} \io u^pw + \io w^q + C_5C_6\io v^r\qquad \text{on } (0,\Tmax), 
\end{align}
where $C_5:=\f {p-1}{p} \cdot (\f{2p}{p+3})^{-\f1{p-1}}$, $C_6:=\f {q-1}{q} \cdot q^{-\f1{q-1}}$.\\
Having chosen $q$ such that $r\in(1+2+\f2n)$, we are able to 
 make use of Lemma \ref{lem:w:lq} and Lemma \ref{lem:v:lr} to estimate the last two terms of \eqref{eq:est:vr&wq}. That, by letting $\varepsilon_1:= \f2{p+3}$ and $\varepsilon_2:= \f1{2C_5C_6}$ and denoting $C_4:=C(\varepsilon_1) + C_5C_6C(\varepsilon_2)$, yields \eqref{lem:pre:v}.
\end{proof}

\subsection{Estimating  $\io |∇w|^2$}
The third ingredient for the final inequality is the following:
\begin{lemma}
Let $\Om\sub ℝ^n$, $n\in ℕ$, be a bounded, smooth domain. Let $(u_0,v_0,w_0)$ satisfy \eqref{init}. There is $C_7>0$ such that the solution of \eqref{sys} satisfies
\begin{equation}\label{lem:pre:w}
\ddt\io|\nabla w|^2 + \io|\nabla w|^2 + \io|D^2w|^2 \leq  \io|\nabla v|^2 +C_7
\end{equation}
on $(0,\Tmax)$
\end{lemma}

\begin{proof}
Recalling the identity $\Delta |\nabla v|^2=2\nabla v \cdot \nabla \Delta v + 2|D^2v|^2$, using Young's inequality, we differentiate the third equation \eqref{sys:w} to obtain
\begin{align}\label{lem:est:w:test}
\f12 \ddt \io |∇w|^2 &=\io \nabla w \cdot \nabla (\Delta w - w -v) \nn \\
&= \f12 \io Δ|∇w|^2 - \io |D^2w|^2- \io |∇w|^2 - \io ∇v\cdot ∇w  \nn  \\
&\leq \f12 \int_{\partial\Omega} \f{\partial |\nabla w|^2}{\partial \nu}- \io |D^2w|^2- \f12\io |∇w|^2 + \f12 \io |∇v|^2  \qquad \text{on } (0,\Tmax)
\end{align}
Aiming at controlling the boundary integral, we first use the one-sided pointwise inequality
\begin{align}\label{lem:est:w:point}
\f {\partial |\nabla w|^2}{\partial \nu} \leq C_8 |\nabla w|^2 \qquad \text {on } \partial\Om
\end{align}
which with some domain-dependent constant $C_8>0$ holds for any bounded smooth domain $\Om\sub ℝ^n$, $n\in ℕ$, and $w$ satisfying $\f{\partial w}{\partial \nu} =0$ on $\partial \Omega$,(see \cite[Lemma 4.2]{mizoguchi_souplet}). With the help of the embeddings $W^{\f12,2}(\Om)\hookrightarrow L^2(\partial \Omega)$ (see \cite[Thm. 9.4]{lions_magenes}) and $W^ {1,2}(\Omega)\hookrightarrow \hookrightarrow W^{\f12,2}(\Omega)\hookrightarrow L^2(\Omega)$(see \cite[Thm. 16.1]{lions_magenes}) and Ehrling's lemma, for any $\varepsilon_3 >0$ one can find $C_9(\varepsilon_3)>0$ such that
\begin{align*}
\int_{\partial \Omega} \varphi^2 \leq \varepsilon_3 \io |\nabla \varphi|^2 + C_9(\varepsilon_3)\io \varphi^2 \qquad \text{for all } \varphi \in W^ {1,2}(\Omega).
\end{align*}
Applying this to $\varepsilon_3:= \f1{2C_8}$ and $\varphi:=\nabla w$, we deduce that
\begin{align}\label{lem:est:w:bdy}
 \f{C_8}2 \int_{\partial \Omega} |\nabla w|^2 \leq \f14 \io |D^2w|^2 + \f{C_8 C_9(\varepsilon_3)}2 \io |\nabla w|^2\qquad \text{on } (0,\Tmax).
\end{align}

In fact, the term $\io |\nabla w|^2$ in \eqref{lem:est:w:bdy} can be controlled by $\io |D^2w|^2$: Invoking Lemma \ref{lem:improve:w} together with Corollary \ref{cor:v:l1}, the Gagliardo-Nirenberg inequality and Young's inequality, we have that for $\varepsilon_4:=\f1{2 C_8 C_9(\varepsilon_3)}>0$, there exists $C_{10}(\varepsilon_4)$ such that
\begin{align}\label{lem:est:w:eps}
\io |\nabla w|^2 \leq \varepsilon_4 \io |D^2 w|^2 + C_{10}(\varepsilon_4) \qquad \text{on } (0,\Tmax).
\end{align}
Inserting \eqref{lem:est:w:point}, \eqref{lem:est:w:bdy} and \eqref{lem:est:w:eps} into \eqref{lem:est:w:test}, we obtain
\begin{align*}
\f12\ddt\io|\nabla w|^2 \leq &-\f12\io|\nabla w|^2 - \left(1-\f14 -\f{C_8 C_9(\varepsilon_3)\varepsilon_4 }2\right) \io|D^2w|^2 \\ &+ \f12\io|\nabla v|^2 + \f{C_8C_9(\varepsilon_3)C_{10}(\varepsilon_4)}2 \qquad \text{on } (0,\Tmax),
\end{align*}
where 
 denoting $C_{7}:=C_8 C_9(\varepsilon_3) C_{10}(\varepsilon_4)$, we conclude \eqref{lem:pre:w}.
\end{proof}

\subsection{The quasi-energy inequality. Proof of Theorem \ref{thm:gebd}}
Without further ado, let us collect the results of the previous subsections:

\begin{lemma}\label{lem:pre:energy}
Let $\Om\sub ℝ^n$, $n\in ℕ$, be a bounded, smooth domain. Let $(u_0,v_0,w_0)$ satisfy \eqref{init}. Let $2\alpha \geq p > 1+\f{n^2}{3n+2}$. There exists $C_{11}>0$ such that the solution of \eqref{sys} satisfies
\begin{equation}\label{eq:pre:energy}
\ddt\Big\{\f1p \io u^p + \f{p+3}4 \io v^2 + \io |\nabla w|^2 \Big\} + \Big\{\f1p \io u^p + \f{p+3}2 \io v^2 + \io |\nabla w|^2 \Big\} \leq C_{11}
\end{equation}
on $(0,\Tmax)$.
\begin{proof}
 \eqref{eq:pre:energy} results from a simple linear combination of \eqref{lem:pre:u}, \eqref{lem:pre:v}, \eqref{lem:pre:w}, where $C_{11}:= C_1 + C_4\cdot\f{p+3}2 + C_7$.
\end{proof}
\end{lemma}

We are now in a position to prove Theorem \ref{thm:gebd}.

\begin{proof}[Proof of Theorem \ref{thm:gebd}]
Upon an ODE comparison principle, \eqref{eq:pre:energy} yields that there exists $C_{12}>0$ such that the solution of \eqref{sys} satisfies
\begin{equation}\label{bd:u:lp}
\io u^p(\cdot,t) \leq C_{12} \qquad \text{for all t } \in (0,\Tmax).
\end{equation}
In view of the extensibility criterion in Lemma \ref{lem:improve:locex}, \eqref{bd:u:lp} asserts that $\Tmax = \infty$. We shall further check the restrictions \eqref{cond:alpha} on $\alpha$, which indeed imply that $2\alpha > \max \{1+\f{n^2}{3n+2},\f{n}2, \f{(1-α)_+n}2, \f{1+(1-α)_+}{1+\f1n}\}$ for $n\in ℕ$. Therefore, it is possile to find some suitable $p$ to make Lemma \ref{lem:pre:energy} and Corollary \ref{cor:u:lp:gives:boundedness} applicable at the same time so that the global boundedness properties therein hold.
\end{proof}


\section*{Acknowledgements}
The authors gratefully acknowledge support of {\em China Scholarship Council} and 
{\em Deut\-scher Aka\-de\-mi\-scher Aus\-tausch\-dienst} 
within the project {\em Qualitative analysis of models for taxis mechanisms}.

{\footnotesize
\bibliographystyle{abbrv}

\begin{thebibliography}{10}

\bibitem{beddington}
J.~R. Beddington.
\newblock Mutual interference between parasites or predators and its effect on
  searching efficiency.
\newblock {\em The Journal of Animal Ecology}, pages 331--340, 1975.

\bibitem{BBTW}
N.~Bellomo, A.~Bellouquid, Y.~Tao, and M.~Winkler.
\newblock Toward a mathematical theory of {K}eller-{S}egel models of pattern
  formation in biological tissues.
\newblock {\em Math. Models Methods Appl. Sci.}, 25(9):1663--1763, 2015.

\bibitem{tao_bellomo}
N.~Bellomo and Y.~Tao.
\newblock Stabilization in a chemotaxis model for virus infection.
\newblock 2017.
\newblock preprint.

\bibitem{bonhoeffer_MSN}
S.~Bonhoeffer, R.~M. May, G.~M. Shaw, and M.~A. Nowak.
\newblock Virus dynamics and drug therapy.
\newblock {\em Proceedings of the National Academy of Sciences},
  94(13):6971--6976, 1997.

\bibitem{deAngelis_et_al}
D.~L. DeAngelis, R.~Goldstein, and R.~O'neill.
\newblock A model for tropic interaction.
\newblock {\em Ecology}, 56(4):881--892, 1975.

\bibitem{hethcote}
H.~W. Hethcote.
\newblock The mathematics of infectious diseases.
\newblock {\em SIAM Rev.}, 42(4):599--653, 2000.

\bibitem{hillen_painter}
T.~Hillen and K.~J. Painter.
\newblock A user's guide to {PDE} models for chemotaxis.
\newblock {\em J. Math. Biol.}, 58(1-2):183--217, 2009.

\bibitem{horstmann}
D.~Horstmann.
\newblock From 1970 until present: the {K}eller-{S}egel model in chemotaxis and
  its consequences. {I}.
\newblock {\em Jahresber. Deutsch. Math.-Verein.}, 105(3):103--165, 2003.

\bibitem{horstmann_winkler}
D.~Horstmann and M.~Winkler.
\newblock Boundedness vs. blow-up in a chemotaxis system.
\newblock {\em J. Differential Equations}, 215(1):52--107, 2005.

\bibitem{keller_segel_70}
E.~F. Keller and L.~A. Segel.
\newblock Initiation of slime mold aggregation viewed as an instability.
\newblock {\em Journal of Theoretical Biology}, 26(3):399--415, 1970.

\bibitem{korobeinikov}
A.~Korobeinikov.
\newblock Global properties of basic virus dynamics models.
\newblock {\em Bulletin of Mathematical Biology}, 66(4):879--883, 2004.

\bibitem{ctnslogsource}
J.~Lankeit.
\newblock Long-term behaviour in a chemotaxis-fluid system with logistic
  source.
\newblock {\em Math. Models Methods Appl. Sci.}, 26(11):2071--2109, 2016.

\bibitem{lankeit_wang}
J.~Lankeit and Y.~Wang.
\newblock Global existence, boundedness and stabilization in a high-dimensional
  chemotaxis system with consumption.
\newblock {\em Discrete and Continuous Dynamical Systems}, 37(12):6099--6121,
  2017.

\bibitem{lions_magenes}
J.-L. Lions and E.~Magenes.
\newblock {\em Non-homogeneous boundary value problems and applications. {V}ol.
  {I}}.
\newblock Springer-Verlag, New York-Heidelberg, 1972.
\newblock Translated from the French by P. Kenneth, Die Grundlehren der
  mathematischen Wissenschaften, Band 181.

\bibitem{mizoguchi_souplet}
N.~Mizoguchi and P.~Souplet.
\newblock Nondegeneracy of blow-up points for the parabolic {K}eller-{S}egel
  system.
\newblock {\em Ann. Inst. H. Poincar\'e Anal. Non Lin\'eaire}, 31(4):851--875,
  2014.

\bibitem{book_nowak}
M.~Nowak and R.~M. May.
\newblock {\em Virus dynamics: mathematical principles of immunology and
  virology: mathematical principles of immunology and virology}.
\newblock Oxford University Press, UK, 2000.

\bibitem{stancevic_amh}
O.~Stancevic, C.~Angstmann, J.~M. Murray, and B.~I. Henry.
\newblock Turing patterns from dynamics of early hiv infection.
\newblock {\em Bulletin of mathematical biology}, 75(5):774--795, 2013.

\bibitem{tao_consumption}
Y.~Tao.
\newblock Boundedness in a chemotaxis model with oxygen consumption by
  bacteria.
\newblock {\em J. Math. Anal. Appl.}, 381(2):521--529, 2011.

\bibitem{wang_ma_lai}
W.~Wang, W.~Ma, and X.~Lai.
\newblock A diffusive virus infection dynamic model with nonlinear functional
  response, absorption effect and chemotaxis.
\newblock {\em Communications in Nonlinear Science and Numerical Simulation},
  42:585 -- 606, 2017.

\bibitem{win_aggregationvs}
M.~Winkler.
\newblock Aggregation vs. global diffusive behavior in the higher-dimensional
  {K}eller-{S}egel model.
\newblock {\em J. Differential Equations}, 248(12):2889--2905, 2010.

\end{thebibliography}

}
\

\end{document}